\documentclass[11pt]{article}
\usepackage{enumerate}
\usepackage{amssymb,a4wide,latexsym,makeidx,epsfig,fleqn}
\usepackage{float}
\usepackage{amsthm}
\usepackage{amsmath}
\usepackage{enumerate}
\usepackage{color}
\usepackage{graphicx}
\usepackage{multirow}
\newtheorem{theorem}{Theorem}[section]
\newtheorem{lemma}[theorem]{Lemma}
\newtheorem{corollary}[theorem]{Corollary}

\newtheorem{example}[theorem]{Example}
\newtheorem{definition}[theorem]{Definition}
\newtheorem{conjecture}[theorem]{Conjecture}

\allowdisplaybreaks

\begin{document}
\textwidth 150mm \textheight 225mm
\title{On the $A_\alpha$ spectral radius and $A_\alpha$ energy of non-strongly connected digraphs \thanks{\small Supported by the National Natural Science Foundation of China (No. 11871398), and the Natural Science Foundation of Shaanxi Province (No. 2020JQ-107).}}

\author{{Xiuwen Yang$^{a,b}$, Ligong Wang$^{a,b,}$\footnote{Corresponding author.}}\\
{\small $^{a}$  School of Mathematics and Statistics, Northwestern Polytechnical University,}\\{\small  Xi'an, Shaanxi 710129, P.R. China.}\\ {\small $^{b}$ Xi'an-Budapest Joint Research Center for Combinatorics, Northwestern Polytechnical University,}\\{\small  Xi'an, Shaanxi 710129, P.R. China.}\\{\small E-mail: yangxiuwen1995@163.com, lgwangmath@163.com}}
\date{}
\maketitle
\begin{center}
\begin{minipage}{135mm}
\vskip 0.3cm
\begin{center}
{\small {\bf Abstract}}
\end{center}
{\small Let $A_\alpha(G)$ be the $A_\alpha$-matrix of a digraph $G$ and $\lambda_{\alpha 1}, \lambda_{\alpha 2}, \ldots, \lambda_{\alpha n}$ be the eigenvalues of $A_\alpha(G)$. Let $\rho_\alpha(G)$ be the $A_\alpha$ spectral radius of $G$ and $E_\alpha(G)=\sum_{i=1}^n \lambda_{\alpha i}^2$ be the $A_\alpha$ energy of $G$ by using second spectral moment. Let $\mathcal{G}_n^m$ be the set of non-strongly connected digraphs with order $n$, which contain a unique strong component with order $m$ and some directed trees which are hung on each vertex of the strong component. In this paper, we characterize the digraph which has the maximal $A_\alpha$ spectral radius and the maximal (minimal) $A_\alpha$ energy in $\mathcal{G}_n^m$.
 }
\end{minipage}
\end{center}

\section{Introduction }

Let $G=(\mathcal{V}(G),\mathcal{A}(G))$ be a digraph which $\mathcal{V}(G)=\{v_1,v_2,\ldots,v_n\}$ is the vertex set of $G$ and $\mathcal{A}(G)$ is the arc set of $G$. For an arc from vertex $v_i$ to $v_j$, we denote by $(v_i, v_j)$, and $v_i$ is the tail of $(v_i, v_j)$ and $v_j$ is the head of $(v_i, v_j)$. The outdegree $d_i^+=d_G^+(v_i)$ of $G$ is the number of arcs whose tail is vertex $v_i$ and the indegree $d_i^-=d_G^-(v_i)$ of $G$ is the number of arcs whose head is vertex $v_i$. We denote the maximum outdegree and the maximum indegree of $G$ by $\Delta^+(G)$ and $\Delta^-(G)$, respectively. A walk $\pi$ of length $l$ from vertex $u$ to vertex $v$ is a sequence of vertices $\pi$: $u=v_0,v_1,\ldots,v_l=v$, where $(v_{k-1},v_k)$ is an arc of $G$ for any $1\leq k\leq l$. If $u=v$ then $\pi$ is called a closed walk. Let $c_2$ denote the number of all closed walks of length $2$. A directed path $P_{n}$ with $n$ vertices is a digraph which the vertex set is $\{v_i|\ i=1,2,\ldots,n\}$ and the arc set is $\{(v_i,v_{i+1})|\ i=1,2,\ldots,n-1\}$. A directed cycle $C_n$ with $n\geq2$ vertices is a digraph which the vertex set is $\{v_i|\ i=1,2,\ldots,n\}$ and the arc set is $\{(v_i,v_{i+1})|\ i=1,\ldots,n-1\}\cup\{(v_n,v_1)\}$. A digraph $G$ is connected if its underlying graph is connected. A digraph $G$ is strongly connected if for each pair of vertices $v_i,v_j\in\mathcal{V}(G)$, there is a directed path from $v_i$ to $v_j$ and one from $v_j$ to $v_i$. A strong component of $G$ is a maximal strongly connected subdigraph of $G$. Throughout this paper, we only consider a connected digraph $G$ containing neither loops nor multiple arcs.

For a digraph $G$ with order $n$, the adjacency matrix $A(G)=(a_{ij})_{n\times n}$ of $G$ is a $(0,1)$-square matrix whose $(i,j)$-entry equals to $1$, if $(v_i, v_j)$ is an arc of $G$ and equals to $0$, otherwise. The Laplacian matrix $L(G)$ and the signless Laplacian matrix $Q(G)$ of $G$ are $L(G)=D^+(G)-A(G)$ and $Q(G)=D^+(G)+A(G)$, respectively, where $D^+(G)=diag(d_1^+,d_2^+,\ldots,d_n^+)$ is a diagonal outdegree matrix of $G$. In 2019, Liu et al. \cite{LiWC} defined the $A_\alpha$-matrix of $G$ as
$$A_\alpha(G)=\alpha D^+(G)+(1-\alpha)A(G),$$
where $\alpha\in[0,1]$. It is clear that if $\alpha=0$, then $A_0(G)=A(G)$; if $\alpha=\frac{1}{2}$, then $A_\frac{1}{2}(G)=\frac{1}{2}Q(G)$; if $\alpha=1$, then $A_1(G)=D^+(G)$. Since $D^+(G)$ is not interesting, we only consider $\alpha\in[0,1)$. The eigenvalue of $A_\alpha(G)$ with largest modulus is called the $A_\alpha$ spectral radius of $G$, denoted by $\rho_\alpha(G)$.

Actually, in 2017, Nikiforov \cite{Ni} first proposed the $A_\alpha$-matrix of a graph $H$ with order $n$ as
$$A_\alpha(H)=\alpha D(H)+(1-\alpha)A(H),$$
where $D(H)=diag(d_1,d_2,\ldots,d_n)$ is a diagonal degree matrix of $H$ and $\alpha\in[0,1]$. After that, many scholars began to study the $A_\alpha$-matrices of graphs. Nikiforov et al. \cite{NiPR} gave several results about the $A_\alpha$-matrices of trees and gave the upper and lower bounds for the spectral radius of the $A_\alpha$-matrices of arbitrary graphs. Let $\lambda_1(A_\alpha(H))\geq\lambda_2(A_\alpha(H))\geq\cdots\geq\lambda_n(A_\alpha(H))$ be the eigenvalues of $A_\alpha(H)$. Lin et al. \cite{LiXS} characterized the graph $H$ with $\lambda_k(A_\alpha(H))=\alpha n-1$ for $2 \leq k\leq n$ and showed that $\lambda_n(A_\alpha(H))\geq2\alpha-1$ if $H$ contains no isolated vertices. Liu et al. \cite{LiDS} presented several upper and lower bounds on the $k$-th largest eigenvalue of $A_\alpha$-matrix and characterized the extremal graphs corresponding to some of these obtained bounds. More results about $A_\alpha$-matrix of a graph can be found in \cite{LiSu,LiHX,LiLX,LiLi,NiRo,WaWT}. Recently, Liu et al. \cite{LiWC} characterized the digraph which had the maximal $A_\alpha$ spectral radius in $\mathcal{G}_n^r$, where $\mathcal{G}_n^r$ is the set of digraphs with order $n$ and dichromatic number $r$. Xi et al. \cite{XiSW} determined the digraphs which attained the maximum (or minimum) $A_\alpha$ spectral radius among all strongly connected digraphs with given parameters such as girth, clique number, vertex connectivity or arc connectivity. Xi and Wang \cite{XiWa} established some lower bounds on $\Delta^+-\rho_\alpha(G)$ for strongly connected irregular digraphs with given maximum outdegree and some other parameters. Ganie and Baghipur \cite{GaBa} obtained some lower bounds for the spectral radius of $A_\alpha(G)$ in terms of the number of vertices, the number of arcs and the number of closed walks of the digraph $G$.

It is well-known that the energy of the adjacency matrix of a graph $H$ first defined by Gutman \cite{Gut} as $E_A(H)=\sum_{i=1}^{n}\nu_i$, where $\nu_i$ is an eigenvalue of the adjacency matrix of $H$. Pe\~{n}a and Rada \cite{PeRa} defined the energy of the adjacency matrix of a digraph $G$ as $E_A(G)=\sum_{i=1}^{n} |\textnormal{Re}(z_i)|$, where $z_i$ is an eigenvalue of the adjacency matrix of $G$ and $\textnormal{Re}(z_i)$ is the real part of eigenvalue $z_i$. Some results about the energy of the adjacency matrices of graphs and digraphs have been obtained in \cite{Bru,CvDS,GuLi}. Lazi\'c \cite{La} defined the Laplacian energy of a graph $H$ as $LE(H)=\sum_{i=1}^n \mu_i^2$ by using second spectral moment, where $\mu_i$ is an eigenvalue of $L(H)$. Perera and Mizoguchi \cite{PeMi} defined the Laplacian energy $LE(G)$ of a digraph $G$ as $LE(G)=\sum_{i=1}^n \lambda_i^2$ by using second spectral moment, where $\lambda_i$ is an eigenvalue of $L(G)$. Yang and Wang \cite{YaWa} defined the signless Laplacian energy as $E_{SL}(G)=\sum_{i=1}^n q_i^2$ of a digraph $G$ by using second spectral moment, where $q_i$ is an eigenvalue of $Q(G)$. In this paper, we study the $A_\alpha$ energy as $E_\alpha(G)=\sum_{i=1}^n \lambda_{\alpha i}^2$ of a digraph $G$ by using second spectral moment, where $\lambda_{\alpha i}$ is an eigenvalue of $A_\alpha(G)$.

\begin{figure}[htbp]
\begin{centering}
\includegraphics[scale=0.75]{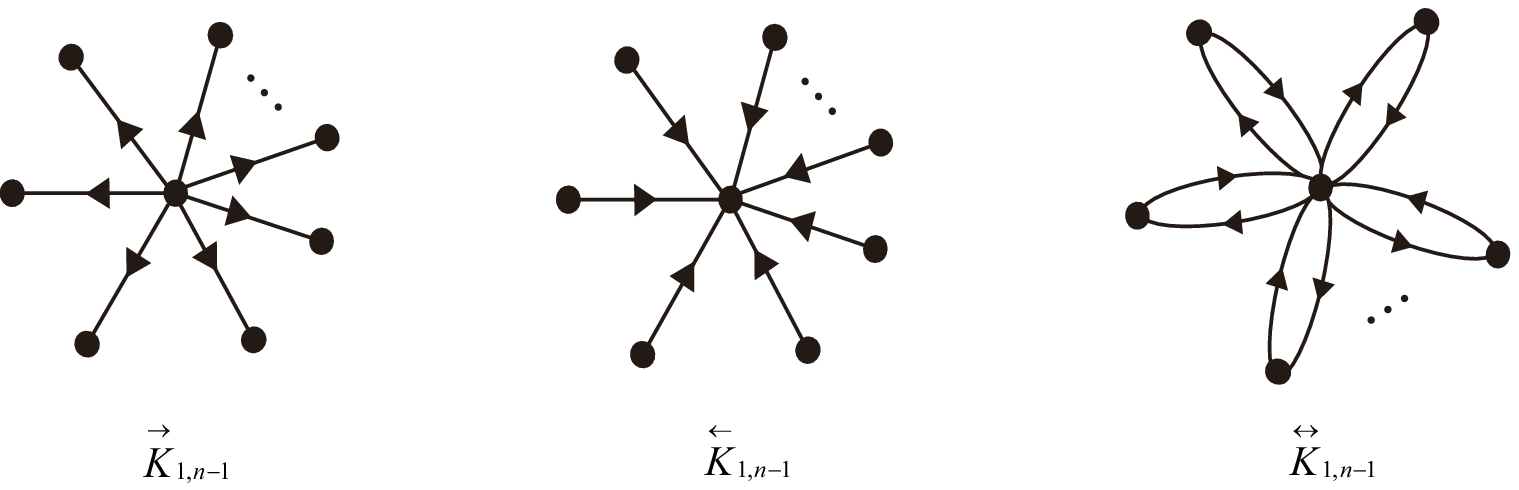}
\caption{An out-star $\stackrel{\rightarrow}{K}_{1,n-1}$, an in-star $\stackrel{\leftarrow}{K}_{1,n-1}$ and a symmetric-star $\stackrel{\leftrightarrow}{K}_{1,n-1}$}\label{fi:ch-1}
\end{centering}
\end{figure}

\begin{figure}[htbp]
\begin{centering}
\includegraphics[scale=1.2]{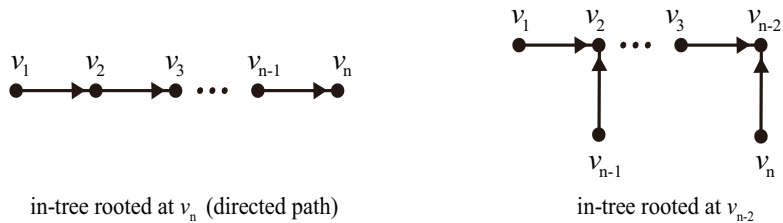}
\caption{Two different in-trees}\label{fi:ch-2}
\end{centering}
\end{figure}

Next, we will introduce some concepts of digraphs.
An arc $(v_i,v_j)$ is said to be simple if $(v_i,v_j)$ is an arc in $G$ but $(v_j,v_i)$ is not an arc in $G$. A digraph $G$ is simple if every arc in $G$ is simple. An arc $(v_i,v_j)$ is said to be symmetric if both $(v_i,v_j)$ and $(v_j,v_i)$ are arcs in $G$. A digraph $G$ is symmetric if every arc in $G$ is symmetric.
Let $T$ be a directed tree with $n$ vertices and $e$ arcs without cycles and $n=e+1$. If $n=1$, then the directed tree is a vertex.
Let $\stackrel{\rightarrow}{K}_{1,n-1}$ be an out-star with $n$ vertices which has one vertex with outdegree $n-1$ and other vertices with outdegree $0$ (see $\stackrel{\rightarrow}{K}_{1,n-1}$ in Figure \ref{fi:ch-1}). And the vertex with outdegree $n-1$ is called the centre of $\stackrel{\rightarrow}{K}_{1,n-1}$.
Let $\stackrel{\leftarrow}{K}_{1,n-1}$ be an in-star with $n$ vertices which has one vertex with indegree $n-1$ and other vertices with indegree $0$ (see $\stackrel{\leftarrow}{K}_{1,n-1}$ in Figure \ref{fi:ch-1}).
Let $\stackrel{\leftrightarrow}{K}_{1,n-1}$ be a symmetric-star with $n$ vertices which all the arcs are symmetric and have a common vertex (see $\stackrel{\leftrightarrow}{K}_{1,n-1}$ in Figure \ref{fi:ch-1}).
Let in-tree be a directed tree with $n$ vertices which the outdegree of each vertex of the directed tree is at most one. Then the in-tree has exactly one vertex with outdegree $0$ and such vertex is called the root of the in-tree (see Figure \ref{fi:ch-2}). Let $\infty[m_1,m_2,\ldots,m_t]$ be a generalized $\infty$-digraph with $n=\sum_{i=1}^t m_i-t+1$ $(m_i\geq2)$ vertices which has $t$ directed cycles $C_{m_i}$ with exactly one common vertex (see $\infty[m_1,m_2,\ldots,m_t]$ in Figure \ref{fi:ch-3}).
A $p$-spindle with $n$ vertices is the union of $p$ internally disjoint $(x,y)$-directed paths for some vertices $x$ and $y$. The vertex $x$ is said to be the initial vertex of spindle and $y$ its terminal vertex. A $(p+q)$-bispindle with $n$ vertices is the internally disjoint union of a $p$-spindle with initial vertex $x$ and terminal vertex $y$ and a $q$-spindle with initial vertex $y$ and terminal vertex $x$. Actually, it is the union of $p$ $(x,y)$-directed paths and $q$ $(y,x)$-directed paths. We denote the $(p+q)$-bispindle by $B[p,q]$ (see $B[p,q]$ in Figure \ref{fi:ch-3}).

\begin{figure}[htbp]
\begin{centering}
\includegraphics[scale=0.95]{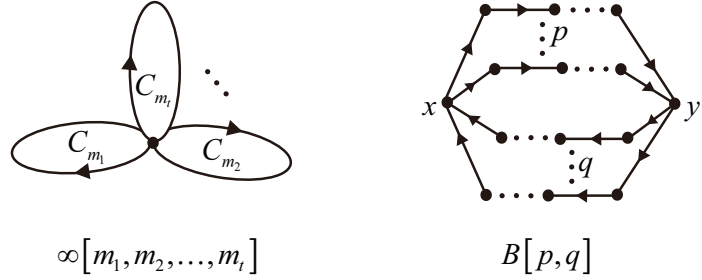}
\caption{A generalized $\infty$-digraph and a $(p+q)$-bispindle}\label{fi:ch-3}
\end{centering}
\end{figure}

\begin{figure}[htbp]
\begin{centering}
\includegraphics[scale=1.0]{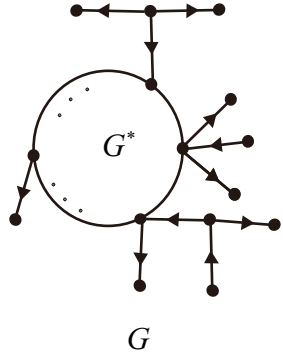}
\caption{A digraph $G\in\mathcal{G}_n^m$}\label{fi:ch-4}
\end{centering}
\end{figure}

Let $\mathcal{G}_n^m$ be the set of non-strongly connected digraphs with order $n$, which contain a unique strong component with order $m$ and some directed trees which are hung on each vertex of the strong component. For a non-strongly connected digraph $G\in\mathcal{G}_n^m$, we assume that $G^*$ is the unique strong component of $G$ with $m$ vertices and $T^{(i)}$ is the directed tree with $n_i$ vertices which hangs on each vertex of $G^*$, where $n=\sum_{i=1}^m n_i$ and $i=1,2,\ldots,m$. Then the vertex set of $G$ is $\mathcal{V}(G)=\bigcup_{i=1}^m \mathcal{V}(T^{(i)})$, where $\mathcal{V}(T^{(i)})=\{u_1^{(i)},u_2^{(i)},\ldots,u_{n_i}^{(i)}\}$, $\mathcal{V}(G^*)=\{v_1,v_2,\ldots,v_m\}$ and $v_i=u_1^{(i)}$, $i=1,2,\ldots,m$. Let $d_{G}^+(u_j^{(i)})$ be the outdegree of vertex $u_j^{(i)}$ of $G$ and $d_{G^*}^+(v_1)\geq d_{G^*}^+(v_2)\geq \cdots\geq d_{G^*}^+(v_m)$ be the outdegrees of vertices of $G^*$, where $i=1,2,\ldots,m$ and $j=1,2,\ldots,n_i$. We take an example in Figure \ref{fi:ch-4}.

\noindent\begin{definition}\label{de:ch-1.1} Let $G\in\mathcal{G}_n^m$ be a non-strongly connected digraph with $n$ vertices.

(i) Let
$$G'=G-\sum_{i=1}^m\sum_{s,t=1}^{n_i} (u_s^{(i)},u_t^{(i)})+\sum_{i=1}^m\sum_{j=2}^{n_i} (u_1^{(i)},u_j^{(i)}),$$
where $(u_s^{(i)},u_t^{(i)})\in\mathcal{A}(G)$, $i=1,2,\ldots,m$ and $s,t,j=1,2,\ldots,n_i$. Then $G'\in\mathcal{G}_n^m$ is a non-strongly connected digraph, which each directed tree $T^{(i)}$ is an out-star $\stackrel{\rightarrow}{K}_{1,n_i-1}$ and the centre of $\stackrel{\rightarrow}{K}_{1,n_i-1}$ is $v_i$ of $G^*$, where $i=1,2,\ldots,m$ (see $G'$ in Figure \ref{fi:ch-5}).

(ii) Let
\begin{align*}
G''&=G-\sum_{i=1}^m\sum_{s,t=1}^{n_i} (u_s^{(i)},u_t^{(i)})+\sum_{i=1}^m\sum_{j=2}^{n_i} (u_1^{(1)},u_j^{(i)})\\
&=G'-\sum_{i=2}^m\sum_{j=2}^{n_i} (u_1^{(i)},u_j^{(i)})+\sum_{i=2}^m\sum_{j=2}^{n_i} (u_1^{(1)},u_j^{(i)}),
\end{align*}
where $(u_s^{(i)},u_t^{(i)})\in\mathcal{A}(G)$, $i=1,2,\ldots,m$ and $s,t,j=1,2,\ldots,n_i$. Then $G''\in\mathcal{G}_n^m$ is a non-strongly connected digraph, which only has an out-star $\stackrel{\rightarrow}{K}_{1,n-m}$ and the centre of $\stackrel{\rightarrow}{K}_{1,n-m}$ is $v_1$ of $G^*$, $v_1$ is the maximal
outdegree vertex of $G^*$ and other directed tree $T^{(i)}$ is just a vertex $v_i$ of $G^*$ for $i=2,3,\ldots,m$ (see $G''$ in Figure \ref{fi:ch-5}).

(iii) Let $G'''\in\mathcal{G}_n^m$ be a non-strongly connected digraph by changing each directed tree in $G$ to an in-tree which the root of the in-tree is $v_i$ of $G^*$, where $i=1,2,\ldots,m$. We take an example in Figure \ref{fi:ch-6}.
\end{definition}

\begin{figure}[htbp]
\begin{centering}
\includegraphics[scale=1.0]{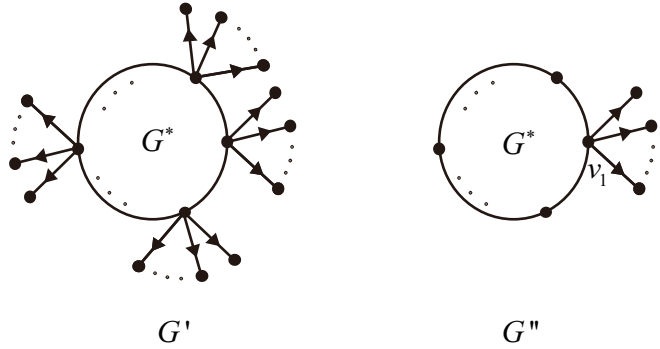}
\caption{Digraphs $G',G''\in\mathcal{G}_n^m$}\label{fi:ch-5}
\end{centering}
\end{figure}

\begin{figure}[htbp]
\begin{centering}
\includegraphics[scale=1.0]{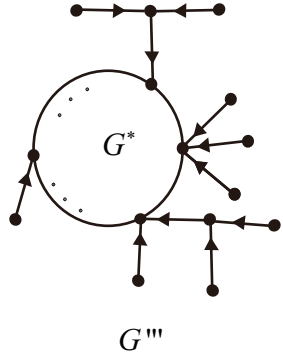}
\caption{A digraph $G'''\in\mathcal{G}_n^m$}\label{fi:ch-6}
\end{centering}
\end{figure}

The arrangement of this paper is as follows. In Section 2, we characterize the digraph which has the maximal $A_\alpha$ spectral radius in $\mathcal{G}_n^m$. In Section 3, we characterize the digraph which has the maximal (minimal) $A_\alpha$ energy in $\mathcal{G}_n^m$.

\section{The maximal $A_\alpha$ spectral radius of non-strongly connected digraphs}

In this section, we will consider the maximal $A_\alpha$ spectral radius of non-strongly connected digraphs in $\mathcal{G}_n^m$. Firstly, we list some known results used for later.

\noindent\begin{lemma}\label{le:ch-2.1}(\cite{HoJo}) Let $M$ be an $n\times n$ nonnegative irreducible matrix with spectral radius
$\rho(M)$ and row sums $R_1,R_2,\ldots,R_n$. Then
$$\underset{1\leq i\leq n}{\mathrm{min}} R_i\leq \rho(M)\leq \underset{1\leq i\leq n}{\mathrm{max}} R_i.$$
Moreover, one of the equalities holds if and only if the row sums of $M$ are all equal.
\end{lemma}

\noindent\begin{definition}\label{de:ch-2.2}(\cite{BePl}) Let $A =(a_{ij})$, $B=(b_{ij})$ be two $n\times n$ matrices. If $a_{ij}\leq b_{ij}$ for all $i$ and $j$, then $A\leq B$. If $A\leq B$ and $A\neq B$, then $A<B$. If $a_{ij}<b_{ij}$ for all $i$ and $j$, then $A\ll B$.
\end{definition}

\noindent\begin{lemma}\label{le:ch-2.3}(\cite{BePl}) Let $A =(a_{ij})$, $B=(b_{ij})$ be two $n\times n$ matrices with the spectral radius $\rho(A)$ and $\rho(B)$, respectively. If $0\leq A\leq B$, then $\rho(A)\leq\rho(B)$. Furthermore, If $0\leq A<B$ and $B$ is irreducible, then $\rho(A)<\rho(B)$.
\end{lemma}

\noindent\begin{lemma}\label{le:ch-2.4}(\cite{LiWC})  Let $G$ be a digraph with the $A_\alpha$ spectral radius $\rho_\alpha(G)$ and maximal outdegree $\Delta^+(G)$. If $H$ is a subdigraph of $G$, then $\rho_\alpha(H)\leq\rho_\alpha(G)$, especially, $\rho_\alpha(G)\geq\alpha\Delta^+(G)$. If $G$ is strongly connected and $H$ is a proper subdigraph of $G$, then $\rho_\alpha(H)<\rho_\alpha(G)$.
\end{lemma}

\noindent\begin{theorem}\label{th:ch-2.5}  Let $G\in\mathcal{G}_n^m$ be a non-strongly connected digraph with $\mathcal{V}(G)=\{v_1,v_2,\ldots,v_n\}$. Let $G^*$ be a unique strong component of $G$ with $\mathcal{V}(G^*)=\{v_1,v_2,\ldots,v_m\}$. Let $\lambda_{\alpha1}, \lambda_{\alpha2}, \ldots, \lambda_{\alpha n}$ be the eigenvalues of $A_\alpha(G)$ and $d_1^+, d_2^+, \ldots, d_n^+$ be the outdegrees of vertices of $G$. Then
$$\lambda_{\alpha i}=\alpha d_i^+,$$
for $i=m+1,m+2,\ldots,n$.
\end{theorem}
\begin{proof}
Let $A_\alpha(G)=\alpha D^+(G)+(1-\alpha)A(G)$ be the $A_\alpha$-matrix of $G$. Let $\mathcal{V}(G)=V_1\bigcup V_2$ be the vertex set of $G$, where $V_1=\mathcal{V}(G^*)=\{v_1,v_2,\ldots,v_m\}$ and $V_2=\mathcal{V}(G-G^*)=\{v_{m+1},v_{m+2},\ldots,v_n\}$. According to the partition of vertex set of $G$, we partition $A_\alpha(G)$ into
$$
A_\alpha(G)=\left[
  \begin{array}{c|c}
   A_{11} & A_{12}\\ \hline
   A_{21} & A_{22}\\
  \end{array}
\right].
$$
The characteristic polynomial $\phi_{A_\alpha(G)}(x)$ of $G$ is $\phi_{A_\alpha(G)}(x)=|xI_n-A_\alpha(G)|$. Since the vertices of $V_2$ are not on the strong component, there must exist a vertex with indegree $0$ or outdegree $0$. Then the elements of column or row of $A_\alpha(G)$ corresponding to that vertex are all $0$, except the diagonal element. So by the property of the determinant, we have $\phi_{A_\alpha(G)}(x)=|xI_n-A_\alpha(G)|=|xI_n-A_{11}|\prod_{i=m+1}^n (x-\alpha d_i^+)$. Hence, $\lambda_{\alpha i}=\alpha d_i^+$, for $i=m+1,m+2,\ldots,n$.
\end{proof}

With the above theorem, we can get a more general result.

\noindent\begin{corollary}\label{co:ch-2.6}
Let $G$ be any digraph with $n$ vertices. Let $\lambda_{\alpha1}, \lambda_{\alpha2}, \ldots, \lambda_{\alpha n}$ be the eigenvalues of $A_\alpha(G)$ and $d_1^+, d_2^+, \ldots, d_n^+$ be the outdegrees of vertices of $G$. For any vertex $v_i$ which is not on the strong components of $G$, we have
$$\lambda_{\alpha i}=\alpha d_i^+.$$
\end{corollary}

Next, we give our main results.

\noindent\begin{theorem}\label{th:ch-2.7} Let $G,G'\in\mathcal{G}_n^m$ be two non-strongly connected digraphs as defined in Definition \ref{de:ch-1.1}. Then $\rho_\alpha(G')\geq\rho_\alpha(G)$.
\end{theorem}
\begin{proof}
By the definition of $G'$, we know $G'\in\mathcal{G}_n^m$ is a non-strongly connected digraph, which each directed tree $T^{(i)}$ is an out-star $\stackrel{\rightarrow}{K}_{1,n_i-1}$ and the centre of $\stackrel{\rightarrow}{K}_{1,n_i-1}$ is $v_i$ of $G^*$, where $i=1,2,\ldots,m$. Then $d_{G'}^+(v_i)=d_{G'}^+(u_1^{(i)})=d_{G^*}^+(v_i)+n_i-1$, $d_{G'}^+(u_j^{(i)})=0$, where $i=1,2,\ldots,m$ and $j=2,3,\ldots,n_i$.

First, we consider the $A_\alpha$-eigenvalues of $G'$. From Theorem \ref{th:ch-2.5}, for the vertex $u_j^{(i)}$ which is not on the strong component $G^*$, we have
$$\lambda_{\alpha j}^{(i)}(G')=\alpha d_{G'}^+(u_j^{(i)})=0,$$
where $i=1,2,\ldots,m$ and $j=2,3,\ldots,n_i$. For the vertex $v_i=u_1^{(i)}$ which is on the strong component $G^*$, the $A_\alpha$-eigenvalues $\lambda_{\alpha 1}^{(i)}(G')$ are equal to the eigenvalues of $A_{11}'$, where
$$A_{11}'=\alpha diag\left(d_{G^*}^+(v_1)+n_1-1,d_{G^*}^+(v_2)+n_2-1,\ldots,d_{G^*}^+(v_m)+n_m-1\right)+(1-\alpha)A(G^*).$$
Obviously, $\rho_\alpha(G')=\rho(A_{11}')$.

Next, we consider the $A_\alpha$-eigenvalues of $G$. From Theorem \ref{th:ch-2.5}, for the vertex $u_j^{(i)}$ which is not on the strong component $G^*$, we have
$$\lambda_{\alpha j}^{(i)}(G)=\alpha d_{G}^+(u_j^{(i)}),$$
where $i=1,2,\ldots,m$ and $j=2,3,\ldots,n_i$. For the vertex $v_i=u_1^{(i)}$ which is on the strong component $G^*$, the $A_\alpha$-eigenvalues $\lambda_{\alpha 1}^{(i)}(G)$ are equal to the eigenvalues of $A_{11}$, where
$$A_{11}=\alpha diag\left(d_{G}^+(v_1),d_{G}^+(v_2),\ldots,d_{G}^+(v_m)\right)+(1-\alpha)A(G^*).$$
Hence, $\rho_\alpha(G)=\underset{1\leq i\leq m,2\leq j\leq n_i}{\mathrm{max}}\left\{\rho(A_{11}), \alpha d_{G}^+(u_j^{(i)})\right\}$.

Finally, we prove
$$\rho_\alpha(G')=\rho(A_{11}')\geq\rho_\alpha(G)=\underset{1\leq i\leq m,2\leq j\leq n_i}{\mathrm{max}}\left\{\rho(A_{11}), \alpha d_{G}^+(u_j^{(i)})\right\}.$$
From Lemma \ref{le:ch-2.3}, since
$$d_{G^*}^+(v_i)+n_i-1\geq d_{G}^+(v_i),$$
we have $A_{11}'\geq A_{11}$. Then $\rho(A_{11}')\geq\rho(A_{11})$. From Lemma \ref{le:ch-2.4}, we have
$$\rho_\alpha(G')\geq\alpha \Delta^+(G') \geq \alpha \Delta^+(G) \geq \alpha d_{G}^+(u_j^{(i)}).$$

Therefore, we have $\rho_\alpha(G')\geq\rho_\alpha(G)$.
\end{proof}

\noindent\begin{theorem}\label{th:ch-2.8} Let $G',G''\in\mathcal{G}_n^m$ be two non-strongly connected digraphs as defined in Definition \ref{de:ch-1.1}. If $\alpha\in\left[\frac{d_{G^*}^+(v_1)}{d_{G^*}^+(v_1)+n-m-n_1+1},1\right)$, then $\rho_\alpha(G'')\geq\rho_\alpha(G')$; if $\alpha=0$, then $\rho_\alpha(G'')=\rho_\alpha(G')$.
\end{theorem}
\begin{proof}
By the definition of $G''$, we know $G''\in\mathcal{G}_n^m$ is a non-strongly connected digraph, which only has an out-star $\stackrel{\rightarrow}{K}_{1,n-m}$ and the centre of $\stackrel{\rightarrow}{K}_{1,n-m}$ is $v_1$ of $G^*$, $v_1$ is the maximal
outdegree vertex of $G^*$ and each other directed tree $T^{(i)}$ is just a vertex $v_i$ of $G^*$ for $i=2,3,\ldots,m$. The vertex set $\mathcal{V}(G'')=\mathcal{V}(\stackrel{\rightarrow}{K}_{1,n-m})\bigcup(\mathcal{V}(G^*)-v_1)$, where $\mathcal{V}(\stackrel{\rightarrow}{K}_{1,n-m})=\{u_1^{(1)},u_2^{(1)},\ldots,u_{n-m+1}^{(1)}\}$, $\mathcal{V}(G^*)=\{v_1,v_2,\ldots,v_m\}$ and $u_1^{(1)}=v_1$. Then $d_{G''}^+(v_1)=d_{G^*}^+(v_1)+n-m$, $d_{G''}^+(u_j^{(1)})=0$ and $d_{G''}^+(v_i)=d_{G^*}^+(v_i)$, where $i=2,3,\ldots,m$ and $j=2,3,\ldots,n-m+1$. Since $d_{G^*}^+(v_1)\geq d_{G^*}^+(v_2)\geq \cdots\geq d_{G^*}^+(v_m)$, by Lemma \ref{le:ch-2.4}, we have
$$\rho_\alpha(G'') \geq \alpha \Delta^+(G'')=\alpha\left(d_{G^*}^+(v_1)+n-m\right).$$

From the proof of Theorem \ref{th:ch-2.7}, we have $\rho_\alpha(G')=\rho(A_{11}')$. By Lemma \ref{le:ch-2.1}, we have
$$\underset{1\leq i\leq m}{\mathrm{min}} R_i(A_{11}')\leq \rho(A_{11}')\leq \underset{1\leq i\leq m}{\mathrm{max}} R_i(A_{11}').$$
Then
$$\rho_\alpha(G')\leq \underset{1\leq i\leq m}{\mathrm{max}} \left\{(1-\alpha)d_{G^*}^+(v_i)+\alpha\left(d_{G^*}^+(v_i)+n_i-1\right)\right\}=\underset{1\leq i\leq m}{\mathrm{max}} \left\{d_{G^*}^+(v_i)+\alpha(n_i-1)\right\}.$$
Without loss of generality, let $\underset{1\leq i\leq m}{\mathrm{max}} \left\{d_{G^*}^+(v_i)+\alpha(n_i-1)\right\}=d_{G^*}^+(v_t)+\alpha(n_t-1)$. That is $d_{G^*}^+(v_t)+\alpha(n_t-1)\geq d_{G^*}^+(v_1)+\alpha(n_1-1)$.

If $\alpha\neq0$, we have $n_t\geq\frac{d_{G^*}^+(v_1)-d_{G^*}^+(v_t)+\alpha n_1}{\alpha}$. Next, we prove
$$\rho_\alpha(G'')\geq \alpha\left(d_{G^*}^+(v_1)+n-m\right)\geq d_{G^*}^+(v_t)+\alpha(n_t-1)\geq \rho_\alpha(G').$$
We only need to prove
$$\alpha\left(d_{G^*}^+(v_1)+n-m\right)\geq d_{G^*}^+(v_t)+\alpha\left(\frac{d_{G^*}^+(v_1)-d_{G^*}^+(v_t)+\alpha n_1}{\alpha}-1\right).$$
That is, $1>\alpha\geq\frac{d_{G^*}^+(v_1)}{d_{G^*}^+(v_1)+n-m-n_1+1}$.

If $\alpha=0$, we have $\rho_0(G')=\rho(A(G^*))$ and $\rho_0(G'')=\rho(A(G^*))$, then $\rho_\alpha(G'')=\rho_\alpha(G')$.

Therefore, if $\alpha\in\left[\frac{d_{G^*}^+(v_1)}{d_{G^*}^+(v_1)+n-m-n_1+1},1\right)$, then $\rho_\alpha(G'')\geq\rho_\alpha(G')$; if $\alpha=0$, then $\rho_\alpha(G'')=\rho_\alpha(G')$.
\end{proof}

From Theorems \ref{th:ch-2.7} and \ref{th:ch-2.8}, we have the following theorem.

\noindent\begin{theorem}\label{th:ch-2.9} Among all digraphs in $\mathcal{G}_n^m$, if $\alpha\in\left[\frac{d_{G^*}^+(v_1)}{d_{G^*}^+(v_1)+n-m-n_1+1},1\right)$ or $\alpha=0$, then $G''$ is a digraph which has the maximal $A_\alpha$ spectral radius.
\end{theorem}

We only obtain $G''$ is a digraph having the maximal $A_\alpha$ spectral radius in $\mathcal{G}_n^m$ when $\alpha\in\left[\frac{d_{G^*}^+(v_1)}{d_{G^*}^+(v_1)+n-m-n_1+1},1\right)$ or $\alpha=0$. However, we know $\frac{d_{G^*}^+(v_1)}{d_{G^*}^+(v_1)+n-m-n_1+1}=\frac{d_{G^*}^+(v_1)}{d_{G^*}^+(v_1)+\sum_{i=2}^{m}(n_i-1)}$. The bigger $\sum_{i=2}^{m}(n_i-1)$ is, the smaller $\frac{d_{G^*}^+(v_1)}{d_{G^*}^+(v_1)+n-m-n_1+1}$ is. And if $\sum_{i=2}^{m}(n_i-1)\rightarrow0$, then $n_1-1\rightarrow n-m$ and $G'\rightarrow G''$. So we think the result is true for all $\alpha\in[0,1)$. Therefore, we give the following conjecture.

\noindent\begin{conjecture}\label{co:ch-2.10}  Among all digraphs in $\mathcal{G}_n^m$, for any $\alpha\in[0,1)$, $G''$ is a digraph which has the maximal $A_\alpha$ spectral radius.
\end{conjecture}

\section{The maximal $A_\alpha$ energy of non-strongly connected digraphs}

In this section, we will consider the maximal $A_\alpha$ energy of non-strongly connected digraphs in $\mathcal{G}_n^m$. Firstly, we will introduce some basic concepts of $A_\alpha$ energy of digraphs.

Let $E_\alpha(G)$ be $A_\alpha$ energy of a digraph $G$. By using second spectral moment, Xi \cite{Xi} defined the $A_\alpha$ energy as $E_\alpha(G)=\sum_{i=1}^n \lambda_{\alpha i}^2$, where $\lambda_{\alpha i}$ is an eigenvalue of $A_\alpha(G)$. She also obtained the following result.

\noindent\begin{lemma}\label{le:ch-3.1}(\cite{Xi}) Let $G$ be a connected digraph with $n$ vertices and $c_2$ be the number of all closed walks of length $2$. Let $d_1^+,d_2^+,\ldots,d_n^+$ be the outdegrees of vertices of $G$. Then
$$E_\alpha(G)=\sum_{i=1}^n \lambda_{\alpha i}^2=\alpha^2\sum_{i=1}^n(d_i^+)^2+(1-\alpha)^2c_2.$$
\end{lemma}

Obviously, we can get the following results.

\noindent\begin{theorem}\label{th:ch-3.2} Let $G$ be a connected digraph with $n$ vertices and $e$ arcs. Let $d_1^+,d_2^+,\ldots,d_n^+$ be the outdegrees of vertices of $G$.

(i) If $G$ is a simple digraph, then
$$E_\alpha(G)=\alpha^2\sum_{i=1}^n (d_i^+)^2.$$

(ii) If $G$ is a symmetric digraph, then
$$E_\alpha(G)=\alpha^2\sum_{i=1}^n (d_i^+)^2+(1-\alpha)^2 e.$$
\end{theorem}
\begin{proof}
From Lemma \ref{le:ch-3.1}, the conclusion is obvious.
\end{proof}

 Let $c_2$ be the number of all closed walks of length $2$ of  a digraph. From Theorem \ref{th:ch-3.2}, we have the following results.
\noindent\begin{example}\label{ex:ch-3.3} We give some $A_\alpha$ energies of special digraphs as follows: \\
(1) $E_\alpha(P_n)=\alpha^2(n-1);$\\
(2) $E_\alpha(C_n)=
\begin{cases}
\alpha^2n, & \textnormal{if}\ n>2,\\
2(\alpha^2-2\alpha+1), & \textnormal{if}\ n=2;
\end{cases}$\\
(3) $E_\alpha(\stackrel{\rightarrow}{K}_{1,n-1})=\alpha^2(n-1)^2;$\\
(4) $E_\alpha(\stackrel{\leftarrow}{K}_{1,n-1})=\alpha^2(n-1);$\\
(5) $E_\alpha(\stackrel{\leftrightarrow}{K}_{1,n-1})=\alpha^2 n(n-1)+2(1-\alpha)^2(n-1);$\\
(6) $E_\alpha(\infty[m_1,m_2,\ldots,m_t])=\alpha^2(t^2+n-1)+(1-\alpha)^2c_2;$\\
(7) $E_\alpha(B[p,q])=\alpha^2(p^2+q^2+n-2)+(1-\alpha)^2c_2.$
\end{example}

\noindent\begin{lemma}\label{le:ch-3.4}(\cite{Xi}) Let $T$ be a directed tree with $n$ vertices. Then
$$\alpha^2(n-1)\leq E_\alpha(T)\leq\alpha^2(n-1)^2.$$
Moreover, $E_\alpha(T)=\alpha^2(n-1)$, if and only if $T$ is an in-tree with $n$ vertices; $E_\alpha(T)=\alpha^2(n-1)^2$ if and only if $T$ is an out-star $\stackrel{\rightarrow}{K}_{1,n-1}$.
\end{lemma}

Next, we give our main results.

\noindent\begin{theorem}\label{th:ch-3.5} Let $G,G'\in\mathcal{G}_n^m$ be two non-strongly connected digraphs as defined in Definition \ref{de:ch-1.1}. Then $E_\alpha(G')\geq E_\alpha(G)$ with equality holds if and only if $G\cong G'$.
\end{theorem}
\begin{proof}
By the definition of $G$, we know $G\in\mathcal{G}_n^m$ is a non-strongly connected digraph with order $n$, which contains a unique strong component with order $m$ and some directed trees which are hung on each vertex of the strong component. From Lemma \ref{le:ch-3.4}, we know the maximal $A_\alpha$ energy of $T^{(i)}$ is
$$\left(E_\alpha(T^{(i)})\right)_{\mathrm{max}}=\alpha^2(n_i-1)^2,$$
where $i=1,2,\ldots,m$. Then we have
\begin{align*}
E_\alpha(G)&=\alpha^2\sum_{i=1}^m\sum_{j=1}^{n_i}\left(d_G^+(u_j^{(i)})\right)^2+(1-\alpha)^2c_2\\
&=\alpha^2\sum_{i=1}^m\left(d_{G^*}^+(u_1^{(i)})+d_{T^{(i)}}^+(u_1^{(i)})\right)^2+ \alpha^2\sum_{i=1}^m\sum_{j=2}^{n_i}\left(d_G^+(u_j^{(i)})\right)^2+(1-\alpha)^2c_2\\
&=\alpha^2\sum_{i=1}^m\left(\left(d_{G^*}^+(v_i)\right)^2+\left(d_{T^{(i)}}^+(u_1^{(i)})\right)^2+
2d_{G^*}^+(v_i)d_{T^{(i)}}^+(u_1^{(i)})\right)\\
&\ \ \ \ +\alpha^2\sum_{i=1}^m\sum_{j=2}^{n_i}\left(d_{T^{(i)}}^+(u_j^{(i)})\right)^2+(1-\alpha)^2c_2\\
&=\alpha^2\sum_{i=1}^m\left(d_{G^*}^+(v_i)\right)^2+\alpha^2\sum_{i=1}^m\sum_{j=1}^{n_i}\left(d_{T^{(i)}}^+(u_j^{(i)})\right)^2
+2\alpha^2\sum_{i=1}^md_{G^*}^+(v_i)d_{T^{(i)}}^+(v_i)+(1-\alpha)^2c_2\\
&\leq \alpha^2\sum_{i=1}^m\left(d_{G^*}^+(v_i)\right)^2+\alpha^2\sum_{i=1}^m(n_i-1)^2
+2\alpha^2\sum_{i=1}^md_{G^*}^+(v_i)(n_i-1)+(1-\alpha)^2c_2\\
&=\alpha^2\sum_{i=1}^m\left(d_{G^*}^+(v_i)+(n_i-1)\right)^2+(1-\alpha)^2c_2\\
&=E_\alpha(G').
\end{align*}

The equality holds if and only if
$$\sum_{j=1}^{n_i}\left(d_{T^{(i)}}^+(u_j^{(i)})\right)^2+2d_{G^*}^+(v_i)d_{T^{(i)}}^+(v_i)= (n_i-1)^2+2d_{G^*}^+(v_i)(n_i-1),$$
for all $i=1,2,\ldots,m$. Anyway, the strong component $G^*$ does not change, so $d_{G^*}^+(v_i)$ does not change. That is, $d_G^+(u_1^{(i)})=d_{T^{(i)}}^+(v_i)=n_i-1$, and $d_G^+(u_j^{(i)})=0$, where $i=1,2,\ldots,m$ and $j=2,3,\ldots,n_i$. Then each directed tree $T^{(i)}$ is an out-star $\stackrel{\rightarrow}{K}_{1,n_i-1}$.

Hence, we have $E_\alpha(G')\geq E_\alpha(G)$ with equality holds if and only if $G\cong G'$.
\end{proof}

\noindent\begin{theorem}\label{th:ch-3.6} Let $G',G''\in\mathcal{G}_n^m$ be two non-strongly connected digraphs as defined in Definition \ref{de:ch-1.1}. Then $E_\alpha(G'')\geq E_\alpha(G')$ with equality holds if and only if $G'\cong G''$.
\end{theorem}
\begin{proof}
By the definition of $G''$, we know $G''\in\mathcal{G}_n^m$ is a non-strongly connected digraph which only has an out-star $\stackrel{\rightarrow}{K}_{1,n-m}$ and the centre of $\stackrel{\rightarrow}{K}_{1,n-m}$ is $v_1$ of $G^*$, $v_1$ is the maximal outdegree vertex of $G^*$ and other directed tree $T^{(i)}$ is just a vertex $v_i$ of $G^*$ for $i=2,3,\ldots,m$. Then we have
$$E_\alpha(G'')=\alpha^2\left(d_{G^*}^+(v_1)+n-m\right)^2+\alpha^2\sum_{i=2}^m(d_{G^*}^+(v_i))^2+(1-\alpha)^2c_2.$$
Since
\begin{align*}
E_\alpha(G')&=\alpha^2\sum_{i=1}^m\left(d_{G^*}^+(v_i)+(n_i-1)\right)^2+(1-\alpha)^2c_2\\
&=\alpha^2\left(\sum_{i=1}^m(d_{G^*}^+(v_i))^2+\sum_{i=1}^m(n_i-1)^2+2\sum_{i=1}^md_{G^*}^+(v_i)(n_i-1)\right)+(1-\alpha)^2c_2\\
&\leq\alpha^2\left(\sum_{i=1}^m(d_{G^*}^+(v_i))^2+\left(\sum_{i=1}^m(n_i-1)\right)^2+2\sum_{i=1}^md_{G^*}^+(v_1)(n_i-1)\right)+(1-\alpha)^2c_2\\
&=\alpha^2\left(\sum_{i=1}^m(d_{G^*}^+(v_i))^2+(n-m)^2+2d_{G^*}^+(v_1)(n-m)\right)+(1-\alpha)^2c_2\\
&=\alpha^2\left(d_{G^*}^+(v_1)+n-m\right)^2+\alpha^2\sum_{i=2}^m(d_{G^*}^+(v_i))^2+(1-\alpha)^2c_2\\
&=E_\alpha(G'').
\end{align*}

The equality holds if and only if
$$\sum_{i=1}^m(n_i-1)^2+2\sum_{i=1}^md_{G^*}^+(v_i)(n_i-1)=\left(\sum_{i=1}^m(n_i-1)\right)^2+2\sum_{i=1}^md_{G^*}^+(v_1)(n_i-1).$$
Anyway, the strong component $G^*$ does not change, so $d_{G^*}^+(v_i)$ does not change. That is, $n_i-1=0$ for all $i=2,3,\ldots,m$ and $n_1=n-m+1$. Then the directed tree $T^{(1)}$ is an out-star $\stackrel{\rightarrow}{K}_{1,n-m}$, and each other directed tree is a vertex $v_i$, where $i=2,3,\ldots,m$. Hence, we have $E_\alpha(G'')\geq E_\alpha(G')$ with equality holds if and only if $G'\cong G''$.
\end{proof}

\noindent\begin{theorem}\label{th:ch-3.7} Let $G,G'''\in\mathcal{G}_n^m$ be two non-strongly connected digraphs as defined in Definition \ref{de:ch-1.1}. Then $E_\alpha(G)\geq E_\alpha(G''')$ with equality holds if and only if $G\cong G'''$.
\end{theorem}
\begin{proof}
From Lemma \ref{le:ch-3.4}, we know the minimal $A_\alpha$ energy of $T^{(i)}$ is
$$\left(E_\alpha(T^{(i)})\right)_{\mathrm{min}}=\alpha^2(n_i-1),$$
where $i=1,2,\ldots,m$. Similar to the proof of Theorem \ref{th:ch-3.5}, we can get the result easily. And
$$E_\alpha(G''')=\alpha^2\sum_{i=1}^n(d_{G^*}^+(v_i))^2+\alpha^2(n-m)+(1-\alpha)^2c_2.$$
\end{proof}

From Theorems \ref{th:ch-3.5} and \ref{th:ch-3.7}, we have the following results.

\noindent\begin{theorem}\label{th:ch-3.8} Among all digraphs in $\mathcal{G}_n^m$, $G''$ is the unique digraph which has the maximal $A_\alpha$ energy and $G'''$ is the unique digraph which has the minimal $A_\alpha$ energy.
\end{theorem}

\noindent\begin{corollary}\label{co:ch-3.9} Let $G\in\mathcal{G}_n^m$ be a non-strongly connected digraph with $n$ vertices. Then
\begin{align*}
&\alpha^2\sum_{i=1}^m(d_{G^*}^+(v_i))^2+\alpha^2(n-m)+(1-\alpha)^2c_2\leq E_\alpha(G)\\
&\leq\alpha^2\left(d_{G^*}^+(v_1)+n-m\right)^2+\alpha^2\sum_{i=2}^m(d_{G^*}^+(v_i))^2+(1-\alpha)^2c_2.
\end{align*}
Moreover, the first equality holds if and only if each directed tree is in-tree which the root of the in-tree is $v_i$ of $G^*$, where $i=1,2,\ldots,m$; the second equality holds if and only if $G\in\mathcal{G}_n^m$ only has an out-star $\stackrel{\rightarrow}{K}_{1,n-m}$ and the centre of $\stackrel{\rightarrow}{K}_{1,n-m}$ is $v_1$ of $G^*$, $v_1$ is the maximal outdegree vertex of $G^*$ and each other directed tree $T^{(i)}$ is just a vertex $v_i$ of $G^*$ for $i=2,3,\ldots,m$.
\end{corollary}

From Corollary \ref{co:ch-3.9}, we can get some bounds of $A_\alpha$ energies of special non-strongly connected digraphs.

\noindent\begin{corollary}\label{ex:ch-3.10} The bounds of $A_\alpha$ energies of special non-strongly connected digraphs $\widehat{U}_n^{m}$, $\widehat{\infty}[m_1,m_2,\ldots,m_t]$ and $\widehat{B}[p,q]$.

(i) Let $\widehat{U}_n^{m}\in\mathcal{G}_n^m$ be a unicyclic digraph with order $n$ which contains a unique directed cycle $C_m$ and some directed trees which are hung on each vertex of $C_m$, where $m\geq2$. Then
$$2\alpha^2+\alpha^2(n-2)+2(1-\alpha)^2\leq E_\alpha(\widehat{U}_n^{2})\leq\alpha^2(n-1)^2+\alpha^2+2(1-\alpha)^2,$$
and
$$\alpha^2m+\alpha^2(n-m)\leq E_\alpha(\widehat{U}_n^{m})\leq\alpha^2(n-m+1)^2+\alpha^2(m-1)\ (m\geq3).$$
Moreover, the first equality holds if and only if each directed tree is in-tree which the root of the in-tree is connected with $C_m$; the second equality holds if and only if $\widehat{U}_n^{m}\in\mathcal{G}_n^m$ only has an out-star $\stackrel{\rightarrow}{K}_{1,n-m}$ and the centre of $\stackrel{\rightarrow}{K}_{1,n-m}$ is an any vertex of $C_m$.

(ii) Let $\widehat{\infty}[m_1,m_2,\ldots,m_t]\in\mathcal{G}_n^m$ be a generalized $\widehat{\infty}$-digraph with order $n$ which contains $\infty[m_1,m_2,\ldots,m_t]$ and some directed trees which are hung on each vertex of $\infty[m_1,m_2,\ldots,m_t]$, where $2=m_1\cdots=m_s<m_{s+1}\leq\cdots\leq m_t$, $m=\sum_{i=1}^tm_i-t+1$ and the common vertex of  $t$ directed cycles $C_{m_i}$ is $v$. Then
\begin{align*}
&\alpha^2(m-1+t^2)+\alpha^2(n-m)+2s(1-\alpha)^2\leq E_\alpha(\widehat{\infty}[m_1,m_2,\ldots,m_t])\\
&\leq\alpha^2\left(n-m+t\right)^2+\alpha^2(m-1)+2s(1-\alpha)^2.
\end{align*}
Moreover, the first equality holds if and only if each directed tree is in-tree which the root of the in-tree is connected with $\infty[m_1,m_2,\ldots,m_t]$; the second equality holds if and only if $\widehat{\infty}[m_1,m_2,\ldots,m_t]\in\mathcal{G}_n^m$ only has an out-star $\stackrel{\rightarrow}{K}_{1,n-m}$ and the centre of $\stackrel{\rightarrow}{K}_{1,n-m}$ is $v$.

(iii) Let $\widehat{B}[p,q]\in\mathcal{G}_n^m$ be a digraph with order $n$ vertices which contains $B[p,q]$ and some directed trees which are hung on each vertex of $B[p,q]$, where $\mathcal{V}(B[p,q])=m$ and $p\geq q$. If both $(x,y)$ and $(y,x)$ are arcs in $\widehat{B}[p,q]$, then
\begin{align*}
&\alpha^2(m-2+p^2+q^2)+\alpha^2(n-m)+2(1-\alpha)^2\leq E_\alpha(\widehat{B}[p,q])\\
&\leq\alpha^2\left(n-m+p\right)^2+\alpha^2(m-2+q^2)+2(1-\alpha)^2.
\end{align*}
Otherwise,
$$\alpha^2(m-2+p^2+q^2)+\alpha^2(n-m)\leq E_\alpha(\widehat{B}[p,q])\leq\alpha^2\left(n-m+p\right)^2+\alpha^2(m-2+q^2).$$
Moreover, the first equality holds if and only if each directed tree is in-tree which the root of the in-tree is connected with $B[p,q]$; the second equality holds if and only if $\widehat{B}[p,q]\in\mathcal{G}_n^m$ only has an out-star $\stackrel{\rightarrow}{K}_{1,n-m}$ and the centre of $\stackrel{\rightarrow}{K}_{1,n-m}$ is $x$.
\end{corollary}

\end{document}